\documentclass{article}

\usepackage{arxiv}

\usepackage[utf8]{inputenc} 
\usepackage[T1]{fontenc}    
\usepackage{hyperref}       
\usepackage{url}            
\usepackage{booktabs}       
\usepackage{amsfonts}
\usepackage{amsmath}
\usepackage{amsthm}
\usepackage{mathtools}
\usepackage{nicefrac}       
\usepackage{microtype}      
\usepackage{lipsum}
\usepackage{fancyhdr}
\usepackage{graphicx}
\graphicspath{ {./images/} }

\usepackage{fancyhdr}

\DeclareMathOperator*{\esssup}{ess\,sup}

\newtheorem{thm}{Theorem}[section]
\newtheorem{lem}[thm]{Lemma}
\newtheorem{rem}[thm]{Remark}

\title{Well-posedness results for general reaction-diffusion transport of oxygen in encapsulated cells}

\author{
    Yuma Nakamura \\
  Graduate School of Natural Science and Technology\\
  \texttt{nakamurayuma@stu.kanazawa-u.ac.jp}
   \And
   Kharisma~Surya~Putri \\
  Graduate School of Natural Science and Technology\\
  \texttt{kharismasp99@gmail.com}\\
  \And
Alef~Edou~Sterk\\
  Bernoulli Institute for Mathematics, Computer Science and Artificial Intelligence, University of Groningen,\\
  The Netherlands\\
  \texttt{a.e.sterk@rug.nl}\\
  \And
 Thomas~Geert~de~Jong\\
  Faculty of Mathematics and Physics\\
  \texttt{t.g.de.jong.math@gmail.com}\\
}

\begin{document}
\maketitle

\let\thefootnote\relax
\footnotetext{This work is supported by JST CREST Grant Number JPMJCR2014.} 

\begin{abstract}
In this paper, we provide well-posedness results for nonlinear parabolic PDEs given by reaction-diffusion equations describing the concentration of oxygen in encapsulated cells. The cells are described in terms of a core and a shell, which introduces a discontinuous diffusion coefficient as the material properties of the core and shell differ. In addition, the cells are subject to general nonlinear consumption of oxygen. As no monotonicity condition is imposed on the consumption monotone operator theory cannot be used. Moreover, the discontinuity in the diffusion coefficient bars us to apply classical results on strong solutions. However, by directly applying a Galerkin method we obtain uniqueness and existence of the strong form solution. These results will provide the basis to study the dynamics of cells in critical states.
\end{abstract}


\section{Introduction}

King et al.~\cite{KING19,king2020corrigendum} proposed models that describe reaction-diffusion of oxygen through a protective shell encapsulating a core of donor cells to determine conditions so that hypoxia of the donor cells can be avoided. This geometry introduces a discontinuous diffusion coefficient as the material properties of the core and shell differ. The results of King et al. are restricted to numerical computation of stationary solutions assuming spherical geometries. In \cite{JON2021topological} their results were made rigorous. In \cite{dejong2023reactiondiffusion} corresponding parabolic PDE is studied for general core-shell geometries. It is shown that the PDE is well-posed and that stationary solutions are stable. These last results crucially depend on the monotonicity of the oxygen consumption which are derived from Michaelis-Menten kinetics. However, during critical cell states such as partial death of donor cells, these monotonicity conditions will not be satisfied. Hence, in this paper we consider the PDE for general consumption, i.e.\  consumption is bounded, non-negative, and zero for negative concentrations.

In the classical theory on existence and uniqueness for non-linear reaction-diffusion equations it is typically assumed that the reaction-term has asymptotics similar to an odd degree polynomial \cite{marion1987attractors,robinson2001infinite}. This means that that the reaction-term is unbounded. Consequently, the regularity of the constructed solution will depend on the degree of the leading asymptotics of the reaction-term. However, in our setting the reaction term is bounded which leads to a degenerate setting with respect to classical theory.

We construct the solutions using the Galerkin method (cf.   \cite{marion1987attractors,robinson2001infinite}).  Although the nonlinearity in our setting does not exactly satisfy the classical results in \cite{marion1987attractors,robinson2001infinite} a Galerkin set-up still works. Additionally, we are dealing with a so-called diffraction problem \cite{ladyzhenskaya2013boundary} meaning that the diffusion coefficient is discontinuous. However, it turns out that the bounds on the term with the Laplacian guarantee that the bounds for the Galerkin approximation are not in danger.  Finally, as with diffraction problems, the loss of regularity resulting from the discontinuity will not be visible when considering the well-posedness of weak solutions but only when we consider the well-posedness of strong solutions.

This paper is organized as follows. In \S \ref{sec:problem}, we present the statement of the problem. In \S \ref{sec:prelim} preliminaries are provided. In \S \ref{sec:results}, we first establish the global existence and uniqueness of the weak solutions, \S \ref{sec:weak}, which is followed by the main results: global existence and uniqueness of the strong solutions in \S \ref{sec:strong}. Finally, in \S \ref{sec:conc}, we provide conclusions and some remarks for future work.

\section{Statement of Problem \label{sec:problem}}
    We start with a description of the core-shell geometry. For an integer $N \geq 2$, let $\Omega\subset \mathbb{R}^N$ with $\bar{\Omega}$ compact, $S\coloneqq\partial \Omega$ the boundary of $\Omega$, $\nu\colon S \rightarrow \mathbb{R}^N$ the outward unit vector, and $T > 0$ a constant.
    Let $\Gamma~(\subset \Omega)$ be an $(N-1)$-dimensional surface that divides $\Omega$ into two open domains $\Omega_1$ and~$\Omega_2$, i.e., $\Gamma = \overline{\Omega}_1\cap \overline{\Omega}_2$, $\Omega = \Omega_1 \cup \Omega_2 \cup \Gamma$, and we suppose $\partial \Omega_1 = \Gamma$ and  $\partial \Omega_2 = S \cup \Gamma$, cf. Figure~\ref{fig:domain}. We take $S,\Gamma$ of class $C^2$. 
    \par
    
The governing equations of our problem are given by:
    \begin{align}
        \frac{d u}{d t} - b \Delta u &=  f(u) && \mbox{in~$\Omega_i \times (0, T)$,\ \  $i=1, 2$}, \label{eq:govp}\\
        u &= 0 && \mbox{on~$S \times (0, T)$}, \label{eq:boundp}\\
        [u]_\Gamma &= 0 && \mbox{on~$\Gamma \times (0, T)$}, \label{eq:contup}\\
        \left[ b \nabla u \cdot \nu \right]_{\Gamma} &=0 && \mbox{on~$\Gamma \times (0, T)$}\label{eq:contfluxp}, \\
        u &= u_0 && \mbox{in~$\Omega$, at $t=0$}, \label{eq:ut0}
    \end{align}\label{gov_eq}
    where 
    $b: \overline{\Omega}_1 \cup \Omega_2 \rightarrow \mathbb{R} $ is given by  
        \begin{align*}
            b(x) \coloneqq \left\{
            \begin{aligned}
                & b_1 & \mbox{if\; $x \in \overline{\Omega}_1$}, \\
                & b_2 & \mbox{if\; $x  \in \Omega_2$},
            \end{aligned}
            \right.
        \end{align*}
    with constants~$b_1, b_2 >0$. Here, the discontinuous diffusion term, $b \Delta u$, is called the \textit{diffraction Laplacian}. $u_0\colon\Omega\to\mathbb{R}$ is a given initial value,
    and $[\cdot]_\Gamma$  denotes the difference of  limiting values on $\Gamma$, i.e., let $u_1$ denote the restriction of $u$ to $\Omega_1$ and $u_2$ the restriction to $\Omega_2$ then $[u]_\Gamma = u_2|_\Gamma- u_1|_\Gamma=0$.

Finally, we let $f: L^2(\Omega) \rightarrow L^2(\Omega)$ be Lipschitz and satisfy
\begin{gather}
\begin{aligned}
(u, f(u) )  & \leq  K ,  \; \forall u \in L^2(\Omega) \\
 \|f\|_{L^2(\Omega)} & \leq  K, 
\end{aligned}
\label{cond:gen_con}
\end{gather} 
with $K>0$ 
and $(\cdot, \cdot)$ denoting the inner product on $L^2(\Omega)$.

The equations \eqref{eq:govp}-\eqref{eq:ut0} are for the transformed concentration. The concentration can be retrieved by $v= c_0 - u$ with $v=c_0$ on $S \times (0,T)$, Appendix A in \cite{dejong2023reactiondiffusion}. Assume that the consumption $g(v)\coloneqq f(c_0-v)$ is bounded, non-negative and zero for negative concentrations. Then, \eqref{cond:gen_con} is satisfied. 

    \begin{figure}[ht]
        \centering
        \includegraphics[width=7cm]{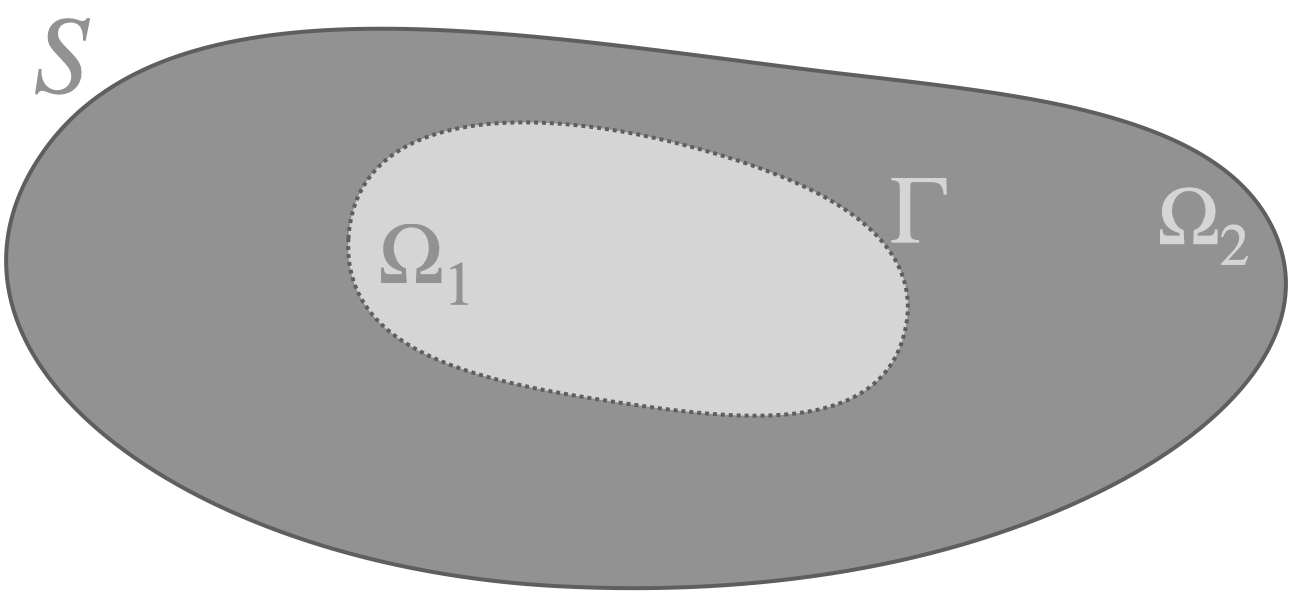}
        \caption{Core-shell geometry of an encapsulated cell.}
        \label{fig:domain}
    \end{figure}

\section{Preliminaries \label{sec:prelim}}
\subsection{Notations}
 We define $V\coloneqq H^1_0(\Omega)$, $H\coloneqq L^2(\Omega)$ and $V^{*}$, $H^{*}$ as their dual space, respectively. The inner product on $V$ is defined by $(u,v)_V = (u,v)_H + (\nabla u, \nabla v)_H$. We denote $(\cdot, \cdot)$ as the inner product on $H$ and  $\langle\cdot,\cdot\rangle$ as the pairing between $V^{*}$ and $V$.  Then, we have   $V \subset \subset H=H^*  \subset V^*$ where we write
 $V\subset\subset H$ to emphasize the compactness of the embedding of $V$ in $H$. Let $X=V,H,V^*$.
 
 The $L^p(0,T;X)$-norm $(p=2,\infty)$ is defined as follows:
 \begin{align*}
    \Vert u \Vert_{L^p(0,T;X)}  \coloneqq \left \{
    \begin{aligned}
        & \left( \int^{T}_0 \Vert u(t) \Vert^2_X dt \right)^{1/2} &&\mbox{if\; $p=2$},\\
        &\esssup_{t\in[0,T]}\|u\|_{X} &&\mbox{if\; $p=\infty$}.
    \end{aligned}
    \right.
\end{align*}

    We define $b_{\max}\coloneqq\max\{b_1,b_2\}$ and $b_{\min}\coloneqq\min\{b_1,b_2\}$. $u_n$ converging weakly to $u$ in $X$ will be denoted by $u_n \rightharpoonup u$ in $X$. 
We reserve $c>0$ to denote generic positive constants that do not depend on the relevant parameters and variables.

\subsection{Diffraction Laplacian}
   We introduce the bilinear form $a: V \times V \rightarrow \mathbb{R}$ given by
    \begin{align*}
        a(u,v) = \int_\Omega  b \nabla u \cdot \nabla v dx.
    \end{align*}
    This induces a linear operator $A: V \rightarrow V^*$ given by
    \begin{align*}
        \langle Au, v \rangle = a(u,v), \qquad  \forall v \in V.
    \end{align*}
    Note that $a(\cdot,\cdot)$ is bounded ($|a(u,v)| \leq c \Vert u \Vert_V \Vert v \Vert_V$) and coercive ($c \Vert u \Vert_V^2  \leq a(u,u)$). Consequently, by Lax--Milgram, $A$ is bijective. Also, observe that $A^{-1}$ is bounded since  for $Au = f \in V^*$, we can write $
\Vert u \Vert^2_V \leq c a(u,u) = c \langle f, u  \rangle  \leq   c \Vert f \Vert_{V^*}  \Vert u \Vert_V,$  which gives $\Vert u \Vert_V   \leq c \Vert f \Vert_{V^{*}}.$

Define $S: H \rightarrow H$ by $S =r \circ A^{-1} \circ\iota$, with $A^{-1}:V^{*}\rightarrow V$, $\iota$ the inclusion map $\iota\colon H\rightarrow V^{*},$ 
and $r$ the restriction map $r\colon V\rightarrow H$.

Since $A$ is bijective and $A^{-1}$ bounded $S$ is compact. Abusing notation we write $A = S^{-1}$ and now consider $A\colon H\rightarrow H$. 
Observe that $A$ is symmetric. From spectral theory for unbounded operators $A$ can be represented by $Au=\sum_{j=1}^{\infty}\lambda_{j}\left(u,w_j\right)w_j,$   where $\lambda_j$ and $w_j$ are the real eigenvalues and eigenfunctions of $A$ respectively. From the smoothness on the boundaries we obtain that the domain of $A$ is given by $D(A) = \{ u \in V \; : \; u|_{\Omega_i} \in H^2(\Omega_i), \; \; u \;\; {\rm satisfies} \;\; \eqref{eq:contfluxp} \}.$ The inner product on $D(A)$ is given by $(u,v)_{D(A)} = (Au,Av)$.

\subsection{Projections}

We define the projection $P_n$ which maps $u\in H$ into the first $n$ eigenfunctions of $A$, $P_n u  \coloneqq \sum_{j=1}^{n} (u,w_j)w_j$. The projection orthogonal to $P_n$ is defined by $Q_n\coloneqq {\rm id} - P_n$.

\subsection{Classical results}

We review the classical results from \cite{robinson2001infinite}.  Most of these results have been reduced to fit the application. A page number is included so that the full statement can be recovered. 

\begin{lem}[p. 199] If $X=H,V,V^*$ then 
\[
\Vert P_n u \Vert_X \leq \Vert u \Vert_X , \qquad P_n u \rightarrow u {\; \rm in \;} X.
\]
\label{lem:7.5}
\end{lem}



\begin{lem}[p. 218] Let $\mathcal{O}$ be a bounded open set in $\mathbb{R}^m$ and let $g_j$ be a sequence of functions in $L^2(\mathcal{O})$ with 
\[
\Vert g_j \Vert_{L^p(\mathcal{O})} \leq C .
\]
If $g \in L^2(\mathcal{O})$ and $g_j \rightarrow g$ pointwise a.e. then $g_j \rightharpoonup g$ in $L^2(\mathcal{O})$.
\label{lem:8.3}
\end{lem}

\begin{thm}[p. 191] Suppose that
\[
u \in L^2(0,T;V) \;\; {\rm and } \;\; \frac{du}{dt} \in L^2(0,T;V^*)
\]
then $u \in  C^0([0,T],  H)$ with the caveat that it may have to be adjusted up to a set of measure zero.
\label{theo:7.2}
\end{thm}

\section{Well-Posedness Results \label{sec:results}}

We will start with the well-posedness of the weak solutions which we then straightforwardly extend to the well-posedness of the strong solution.
\subsection{Weak Solutions \label{sec:weak}}

We consider
\begin{align}
\frac{du}{dt} + Au =  f(u), \label{eq:govw}
\end{align}
as an equality in $L^2(0,T;V^*)$.
    \begin{thm}[Well-Posedness of Weak Solutions] 
        Equation \eqref{eq:govw} with $u(0)=u_0 \in H$ has a unique weak solution $u$ for any $T>0$, 
        \[u \in L^2(0,T;V), \; \; \frac{du}{dt} \in L^2(0,T;V^*),\]
        and $u  \in C^0([0,T];H) $ with the caveat that it may have to be adjusted on a set of measure zero. 
        Furthermore, $u_0 \mapsto u(t)$ is in $C^0(H;H)$. 
        \label{theo:mainweak}
    \end{thm}

    \begin{proof} 
        We consider the solutions expressed by the first $n$ eigenfunctions of $A$:
        \begin{align*}
            u_n(t) = \sum_{j=1}^{n} u_{nj}(t) w_j,
        \end{align*}
satisfying 
\begin{align}
\left( \frac{du_n}{dt} , w_i  \right) + \left(A u_n , w_i \right)  = \left(f(u_n), w_i \right), \qquad 1 \leq i \leq n \label{eq:eigODE}
\end{align}
with $(u_n(0),w_i) = (u_0,w_i)$.  Define $H_n \coloneqq P_n H \subset H$. Hence, we need to solve the IVP
\begin{align}
\frac{dv}{dt} + Av = P_n f(v), \qquad v(0) = P_n u(0), \label{eq:galerODEsystem}
\end{align}
on the finite-dimensional space $H_n$.
The mapping $v \mapsto -Av + P_n f(v)$ is Lipschitz continuous from $H_n$ to $H_n$. By standard existence uniqueness results for ODEs, the system \eqref{eq:galerODEsystem} has a unique solution on some finite interval $[0,T]$ with $T$ dependent on $n$ and $u_0$.  We will see that the solution exists for all $T>0$.
        
        Consider the inner product of \eqref{eq:galerODEsystem} with $u_n$:
        \begin{align}
            \left(\frac{du_{n}}{dt},u_n \right) +  (A u_{n},u_n) &=  (P_n f(u_n),u_n).\label{eq:innerGar}
        \end{align} 
        Observe that $(P_n f(u_n),u_n)  = (f(u_n), P_n u_n) = (f(u_n),u_n)$. From the assumption $(u,f(u))\leq K$ and the coercivity of $a(\cdot,\cdot)$, we obtain the following:
        \begin{align*}
            \frac12 \frac{d \|u_n\|^2_H}{dt} +  b_{\min} \Vert u_n \Vert_V^2  \leq K. 
        \end{align*}
    Integrating both sides over $t$ between $0$ and $T$ gives 
        \begin{align*}
            \frac12 \|u_n(T)\|^2_H +  b_{\min}  \int_{0}^T \Vert u_n \Vert_V^2 dt  \leq K T + \frac1{2} \|u(0)\|^2_H.
        \end{align*}
We define $\gamma \coloneqq K T + \frac1{2} \|u(0)\|^2_H $.  Then, we obtain the following bounds:
\begin{align}
\sup_{t \in [0,T]} \|u_n(t)\|^2_H &   \leq 2 \gamma, \label{eq:uT} \\
\int_{0}^T \Vert u_n \Vert_V^2 dt  &\leq \frac{\gamma}{b_{\min}}  \label{eq:uL2}.
\end{align}
Observe that $\gamma$ is linear in $T$. Hence, Equation \eqref{eq:uT} with local existence of solutions for \eqref{eq:galerODEsystem} gives existence of solutions for any $T>0$. Observe that \eqref{eq:uT} \eqref{eq:uL2}
        give that $u_n$ is uniformly bounded in $L^\infty(0,T; H)$ and $L^2(0,T; V)$.
        
Since $\|f\|_H \leq K$ we obtain that $f(u_n)$ is uniformly bounded in $L^2(0,T; H)$ and $Au_n$ is uniformly bounded in $L^2(0,T; V^*)$. Hence, from \eqref{eq:galerODEsystem}, $du_n/dt$ is uniformly bounded in $L^2(0,T; V^*)$.
        By Aloaglu's compactness theorem we can extract a weakly convergent subsequence $u_n$, with $u_n \rightharpoonup u$ in $L^2(0,T;V)$ and $f(u_n) \rightharpoonup \chi $ in $L^2(0,T;H)$. The strong convergence $u_n \rightarrow u $ in $L^2(0,T;H)$ is obtained by using Lemma \ref{lem:8.3}.

Now we want to show that $P_n f(u_n) \rightharpoonup \chi$ in $L^2(0,T;H)$. We have that
\begin{align}
\int_{\Omega_T}(P_n f(u_n) - \chi) \phi dx dt &= \int_{\Omega_T} (f(u_n) - \chi) \phi dx dt  -  \int_{\Omega_T} Q_n f(u_n)\phi dx dt, \label{eq:Pnfweak}
\end{align}
for all $\phi \in L^2(0,T;H)$. Recall that $f(u_n) \rightharpoonup \chi $ in $L^2(0,T;H)$. So we just need to consider the $Q_n$ term in \eqref{eq:Pnfweak}. Observe that $\| Q_n f(u_n) \|_H  = \| f(u_n) \|_H  \leq K $.  We can consider $\phi = \sum_{j=1}^m \alpha_j(t) \phi_j$ where $\alpha_j \in L^2(0,T)$ and $ \phi_j \in C^\infty_c(\Omega))$ since $\phi$ dense in $ L^2(0,T;H)$.  From Lemma \ref{lem:7.5} we obtain that $Q_n \phi_j  \rightarrow 0$ in $H$ and we have shown that $P_n f(u_n) \rightharpoonup \chi$ in $L^2(0,T;H)$.

        Combining the results we arrive at the equality
        \begin{align*}
            \frac{du}{dt} + Au = \chi,
        \end{align*}
        holds in the dual space $L^2((0,T); V^*)$.
        
        Next, we show that $\chi = f(u)$.
        Since $u_n \rightarrow u$ in $L^2(0,T;H)$  there exists a subsequence $u_{n_j}$ such that $u_{n_j}(x,t)\rightarrow u(x,t)$ for almost every $(x,t) \in [0,T]\times\Omega$.
        Note that  $f(u_{n_j})(x,t) \rightarrow f(u)(x,t)$ for almost every $(x,t) \in [0,T]\times\Omega)$ and $f(u_{n_j})$ is uniformly bounded in $L^2(0,T;H)$. Therefore, by Lemma \ref{lem:8.3}, $f(u_{n_j}) \rightharpoonup f(u)$ in $L^2(0,T;H)$.
        Finally, by the uniqueness of the weak limit we obtain that $\chi = f(u)$.

        Now we have $u \in L^2(0,T;V)$ and $ {du}/{dt} \in L^2(0,T;V^*)$. By Theorem \ref{theo:7.2}, $u \in C^0( [0,T];H)$. 
        
        To show that $u_n(0)= u(0)$, let $\phi \in C^1([0,T];V)$ with $\phi(T)=0$. 
        Consider the limiting equation of the approximation,
        \[
            \left\langle \frac{du}{dt} , v \right\rangle + a(u,v) = \left\langle f(u), v \right\rangle \left(v \in V\right).
        \]
        Integrating from $0$ to $T$ and using integration by parts we get
        \begin{equation}
            \int^T_0 - \langle u, \phi' \rangle + a(u, \phi) dt = \int^T_0 \langle f(u(t)), \phi  \rangle dt + (u(0),\phi(0)). \label{lim_eq}
        \end{equation}
        On the other hand, from the Galerkin approximation, we have 
        \begin{equation}
            \int^T_0 - \langle u_n, \phi' \rangle + a(u_n, \phi) dt = \int^T_0 \langle P_n f(u_n(t)), \phi  \rangle dt  + (u_n(0),\phi(0)). \label{gal_app}
        \end{equation}
        Recall that $u_n(0) = P_n u_0 \rightarrow u_0 $.
        Then, taking the limit in \eqref{gal_app} and comparing with \eqref{lim_eq} we obtain $(u_0 - u(0), \phi(0)) =0 $ which implies $u_0 =u(0)$ as  $\phi(0)$ is arbitrary. 
        
        To show the uniqueness and continuous dependence of the solutions, take $u_0, v_0 \in H$ and consider the corresponding solutions $u,v$. 
We define $w \coloneqq u-v$. Then, $w$ satisfies
\[
\frac{dw}{dt} + A w = f(u)-f(v), \qquad w(0) = u_0-v_0.
\]
We take the inner product with $w$ to obtain $    \frac12 \frac{d\|w\|^2_H}{dt} + (A w,w) = (f(u)-f(v),u-v)$.
Because $(f(u)-f(v),u-v) \leq c{\|u-v\|_{H}^2}$ and $ (A w,w) \geq  b_{\min} \Vert w \Vert_V $, we have $\frac12 \frac{d\|w\|^2_H}{dt} \leq c\|w\|_H^2$. 
By integrating over $t$ we get $\|u(t)-v(t)\|_H \leq \|u_0-v_0\|_H e^{ct}$ which implies the uniqueness and continuous dependence on initial conditions.
\end{proof}

\subsection{Strong Solutions \label{sec:strong}}
In this section, we consider more regular solutions (we refer to such solutions as strong solutions).
It is considered by regarding \eqref{eq:govw} as an equality that holds in $L^2(0,T;H)$. 

    \begin{thm}[Well-Posedness of Strong Solutions] Equation \eqref{eq:govw} with $u(0)=u_0 \in V$ has a unique solution $u$ for any $T>0$ with
        \[
            u \in L^2(0,T;D(A)), \; \; \frac{du}{dt} \in L^2(0,T;H), 
        \] 
         and $u \in C^0([0,T];V)$ with the caveat that it may have to be adjusted on a set of measure zero.
        Furthermore, $u_0 \mapsto u(t)$ is in $C^0(V;V)$. \label{theo:mainstrong}
    \end{thm}


\begin{proof}
    We follow a similar method as the proof of Theorem \ref{theo:mainweak}. Now, we consider taking the inner product of \eqref{eq:galerODEsystem} with $Au_n$, which gives  
    \begin{align}
        \left(\frac{d u_n}{dt} , A u_n \right) + \| A u_n \|_H^2   = (P_n f(u_n) , A u_n). \label{stronginner}
    \end{align}
    
    Now, let us assume that $b_1>b_2$. By integrating the first term on the LHS of \eqref{stronginner} over $t\in[0,T]$, 
    and using the fact that  $\left(\frac{d u_n}{dt} , A u_n \right)\ = a(u_n,\frac{d u_n}{dt} )$, we have
    \begin{equation} \label{eq:evaluate_storong}
    \frac{b_{2}}{2}\|u_n(T)\|_{V}^2-\frac{b_{1}}{2}\|u_n(0)\|^2_{V}\leq\int_{0}^{T}\left(\frac{du_n}{dt},Au_n\right)dt,
    \end{equation}
    On the other hand, by applying the Cauchy-Schwarz inequality and Young's inequality to the RHS of \eqref{stronginner} and combining with \eqref{eq:evaluate_storong}, we obtain the following:
    \begin{equation*}
        \begin{split}
        &\frac{b_{2}}{2}\|u_n(T)\|_{V}^2-\frac{b_{1}}{2}\|u_n(0)\|^2_{V}
        +\| u_n \|_{L^2(0,T;D(A))}^2 
        \leq \frac{1}{2}\|P_nf(u_n)\|_{L^2(0,T;H)}^2+
        \frac{1}{2}\| u_n \|_{L^2(0,T;D(A))}^2\\
        &\Rightarrow 
        b_{2}\|u_n(T)\|_{V}^2-b_{1}\|u_n(0)\|^2_{V}
        \leq \|P_nf(u_n)\|_{L^2(0,T;H)}^2
        \leq \|f(u_n)\|_{L^2(0,T;H)}^2
        \end{split}
    \end{equation*}  
    
By means of a similar argument as in the proof of Theorem \ref{theo:mainweak}, we can show that $u_n \rightarrow u $ in $L^2(0,T;D(A))$ and  $P_n f(u_n) \rightharpoonup f(u)$ in $L^2(0,T;H)$.

We have that $u \in L^2(0,T; D(A))$ and $\frac{du}{dt} \in L^2(0,T; H)$. So if we take $v = (\nabla u)_i$ then $v \in L^2(0,T; V)$ and $\frac{dv}{dt} \in L^2(0,T; V^*)$  which allows us to apply Theorem \ref{theo:7.2}
to obtain that $v \in C^0([0,T]; H)$. Consequently, we have that $u \in C^0([0,T]; V)$.

Next, we adapted the continuous uniqueness proof of Theorem \ref{theo:mainweak} for $V$.
Take $u_0, v_0 \in V$ and consider the corresponding solutions $u,v$. Define $w \coloneqq u-v$.
Then, $w$ satisfies
\[
\frac{dw}{dt} + A w = f(u)-f(v), \qquad w(0) = u_0-v_0.
\]

By taking the inner product with $Aw$ and using the fact that $f$ is Lipschitz continuous, we have
\begin{align*}
\left(\frac{dw}{dt},Aw\right) + \| Aw \|_H^2 \leq (f(u)-f(v),Aw) \leq \frac{1}{2} c^2\|w\|^2_V + \frac{1}{2}  \| Aw \|_H^2  . \label{eq:w2}
\end{align*}

We move the second term in the RHS to LHS, and drop the $\| Aw \|_H^2$ term. By integration over $[0,t]$, and applying the similar steps as in \eqref{eq:evaluate_storong} to the LHS, we get

\begin{equation*}
    \|w(t)\|_{V}^2\leq \frac{c^2}{b_{2}}\int_{0}^{t}\|w(s)\|^2_{V}ds+\frac{b_{1}}{b_{2}}\|w(0)\|_{V}^2.
\end{equation*}
From Gronwall's inequality, 
we obtain $\|u(t)-v(t)\|_{V}^2\leq (b_{1}/b_{2})\|u(0)-v(0)\|^2_{V}\exp(c^2t/b_{2})$.
Therefore, we can prove uniqueness and continuous dependence on initial conditions.
In the case of $b_2>b_1$, the result can be obtained by interchanging $b_2$ and $b_1$.

\end{proof}

\begin{rem}
Let us consider the classical approach to the Laplacian problem with smooth diffusion and regularize the jump in diffusivity by passing to a limit where the transition region between $b_1$ and $b_2$ shrinks to a lower-dimensional surface. 

We state a Laplacian problem with diffusion coefficient $b_{\varepsilon}\colon \bar\Omega \rightarrow \mathbb{R} ({\forall}\varepsilon>0)$, where $b_{\varepsilon}$ is sufficiently smooth.
   Let $u_{\varepsilon}$ be the solution of the following
    \begin{align}
        \frac{d u_{\varepsilon}}{dt}-\nabla\cdot(b_{\varepsilon}\nabla u_{\varepsilon})&=f &&\mbox{ in $\Omega\times[0,T]$}\\
        u_{\varepsilon}&=0 &&\mbox{ in $\partial\Omega\times[0,T]$}\\
        u_{\varepsilon}&=u_{\varepsilon}^{0} &&\mbox{ in $\bar{\Omega}$ at $t=0$}.
    \end{align}
with $f \in L^2(0,T; L^2(\Omega))$ then $u_{\varepsilon}\in L^2(0,T;H^2(\Omega))\cap L^{\infty}(0,T;H^1_0(\Omega))$ for all $\varepsilon>0$, we refer to \cite{evans} for the full statement. We note that the statement can be extended to nonlinear $f$ using the techniques in the proof of Theorem \ref{theo:mainweak}.


When \( b_\varepsilon \to b \) as \( \varepsilon \to 0 \), the regularized solutions \( u_\varepsilon \) are uniformly bounded in \( L^2(0, T; V) \). By the compactness theorem we can extract a subsequence that converges weakly in \( L^2(0, T; V) \). Hence, by passing to the limit \( \varepsilon \to 0 \), the weak limit \( u \) satisfies the weak formulation of the PDE with the discontinuous diffusion coefficient \( b \). Thus, we can obtain \( u \in L^2(0, T; V) \). 
However, the strong solution in \(D(A) \) is not guaranteed due to the potential loss of regularity introduced by the discontinuity in \( b \). This is because the higher regularity (i.e., second order derivatives) may not be controlled uniformly as \( \varepsilon \to 0 \), particularly at the interface where the jump occurs.



In our paper, we tackled the discontinuous diffusion problem directly, involving proving the existence and uniqueness of weak solutions using the Galerkin method and then leveraging elliptic regularity results to demonstrate that these weak solutions are actually strong solutions.
\end{rem}

\section{Conclusion \label{sec:conc}}
  
In this paper, we established the global existence and uniqueness of strong solutions for reaction-diffusion equations with diffraction Laplacian and nonlinear terms describing general oxygen consumption.  These results extend previous work \cite{dejong2023reactiondiffusion} which relied on monotonicity properties of the nonlinear term. This work can be used to make results in \cite{diez2024turing} rigorous as well as provide the theoretical foundation for future numerical work on the dynamics of critical cell states.

\section{Acknowledgement}
The authors would like to thank Georg Prokert and Hirofumi Notsu for their assistance.

\end{document}